\documentclass{amsart}

\usepackage{amsmath,amsthm,mathrsfs,amssymb,graphicx}
\usepackage[all]{xy}
\usepackage{tikz}
\usetikzlibrary{shapes,snakes}
\usepackage[pdfborder={0 0 0},backref=false,colorlinks=true,linktocpage=true]{hyperref}

\newtheorem{theorem}{Theorem}[section]
\newtheorem{proposition}[theorem]{Proposition}
\newtheorem{lemma}[theorem]{Lemma}
\newtheorem{corollary}[theorem]{Corollary}
\theoremstyle{definition}
\newtheorem{definition}[theorem]{Definition}

\newcommand{\res}{\mathbin{\upharpoonright}}

\newcommand{\seq}[1]{\langle #1 \rangle}

\newcommand{\forces}{\Vdash}

\begin{document}

\title{Limits to joining with generics and randoms}

\author{Adam R. Day}
\address{Department of Mathematics\\
University of California, Berkeley\\
Berkeley, California 94720 U.S.A.}
\email{adam.day@math.berkeley.edu}

\author{Damir D. Dzhafarov}
\address{Department of Mathematics\\
University of California, Berkeley\\
Berkeley, California 94720 U.S.A.}
\email{damir@math.berkeley.edu}

\thanks{The authors are grateful to Theodore Slaman for posing the question that motivated this article as well as for insightful discussions. The first author was supported by a Miller Research Fellowship in the Department of Mathematics at the University of California, Berkeley. The second author was supported by an NSF Postdoctoral Fellowship.}

\maketitle

\begin{abstract}
Posner and Robinson \cite{PS-1981} proved that if $S \subseteq \omega$ is non-computable, then there exists a $G \subseteq \omega$ such that $S \oplus G \geq_T G'$. Shore and Slaman \cite{SS-1999} extended this result to all $n \in \omega$, by showing that if $S \nleq_T \emptyset^{(n-1)}$ then there exists a $G$ such that $S \oplus G \geq_T G^{(n)}$. Their argument employs Kumabe-Slaman forcing, and so the set they obtain, unlike that of the Posner-Robinson theorem, is not  generic for Cohen forcing in any way. We answer the question of whether this is a necessary complication by showing that for all $n \geq 1$, the set $G$ of the Shore-Slaman theorem cannot be chosen to be even weakly $2$-generic. Our result applies to several other effective forcing notions commonly used in computability theory, and we also prove that the set $G$ cannot be chosen to be $2$-random.
\end{abstract}

\section{Introduction}\label{S:introduction}
Our starting point is the following well-known theorem of computability theory.

\begin{theorem}[Posner and Robinson \cite{PS-1981}, Theorem 1]\label{T:Posner_Robinson}
For all $S \subseteq \omega$, if $S$ is non-computable there exists $G \subseteq \omega$ such that $S \oplus G, S \oplus \emptyset' \geq_T G'$. (In particular, when $S \leq_T \emptyset'$ then $S \oplus G \equiv_T G' \equiv_T \emptyset'$.)
\end{theorem}

In order to generalize the  Posner-Robinson theorem  to higher numbers of jumps, we need an additional condition on $S$. If we want $S \oplus G \geq_T G^{(n)}$, then we  need $S$ to be not computable in 
$\emptyset^{(n-1)}$, since otherwise 
$S \oplus G \leq_T \emptyset^{(n-1)} \oplus G \leq_TG^{(n-1)}$. 
In connection with their work on the definability of the jump, Shore and Slaman proved this is the only condition needed for the generalization.

\begin{theorem}[Shore and Slaman \cite{SS-1999}, Theorem 2.1]\label{T:Slaman_Shore}
For all $n \geq 1$ and all $S \subseteq \omega$, if $S \nleq_T \emptyset^{(n-1)}$ there exists $G \subseteq \omega$ such that $S \oplus G, S \oplus \emptyset^{(n)} \geq_T G^{(n)}$.
\end{theorem}

The proof of Shore-Slaman theorem is not a simple generalization of  Posner and Robinson's proof. 
While both theorems are proved using forcing constructions,  they differ in the underlying forcing notion used. 
 Posner and Robinson use Cohen forcing, and show that the set $G$ in Theorem \ref{T:Posner_Robinson} can actually be chosen to be Cohen $1$-generic.  Shore and Slaman use a considerably more intricate notion of forcing, due to Kumabe and Slaman (discussed further in Section \ref{S:forcing}). This raises the question of whether Theorem \ref{T:Slaman_Shore} \emph{can} be proved using Cohen forcing. 
  More generally, one can ask whether Theorem \ref{T:Slaman_Shore} can be proved using some other commonly used forcing notion, such as Jockusch-Soare forcing or Mathias forcing.

In this article, we give a negative answer to the above questions. We prove a general result that applies to most of the forcing notions $\mathbb{P}$ used in computability theory, and which roughly states that if $n$ is sufficiently large and $G \subseteq \omega$ is sufficiently generic for $\mathbb{P}$, then $G$ does not satisfy Theorem \ref{T:Slaman_Shore}. For Cohen forcing, $n = 1$ and weak $2$-genericity suffice. We then prove a similar result for randomness, showing that for each $n \ge 2$, the set $G$ in Theorem \ref{T:Slaman_Shore} cannot be chosen to be $n$-random.
The conclusion is that the complexity inherent in Kumabe-Slaman forcing is in fact essential to proof of Theorem \ref{T:Slaman_Shore}.

We refer the reader to Soare \cite{Soare-TA} and Downey and Hirschfeldt \cite{DH-2012} for background on computability theory and algorithmic randomness, respectively. A brief account of forcing in arithmetic, as we shall use it, is included in Section \ref{S:forcing} below.

\section{A non-joining theorem for generics}\label{S:generics}

The purpose of this section is to prove the following theorem. We shall prove in the next section a  theorem applying to forcing notions in general, of which this will be the special case for Cohen forcing.
We present this argument separately in order to make the basic idea easier to understand.

\begin{theorem}\label{T:Cohen_gen}
There exists a $\emptyset'$-computable perfect tree $T \subseteq 2^{<\omega}$ such that for all $S \in [T]$ and all $G \subseteq \omega$, if $G$ is weakly $2$-generic then $S \oplus G \ngeq_T \emptyset'$.
\end{theorem}

The theorem establishes that for all $n \geq 2$, the set $G$ in Theorem \ref{T:Slaman_Shore} cannot be chosen to be $n$-generic.

\begin{corollary}\label{C:Cohen_gen}
For all $n \geq 2$, there exists an $S \nleq_T \emptyset^{(n-1)}$ such that for all $G \subseteq \omega$, if $G$ is weakly $2$-generic then $S \oplus G \ngeq_T \emptyset'$. In particular, $S \oplus G \ngeq_T G^{(n)}$.
\end{corollary}

\begin{proof} Let $T \subseteq 2^{<\omega}$ be the tree obtained from Theorem~\ref{T:Cohen_gen}. As $T$ is perfect, we may choose an $S \in [T]$ such that $S \nleq_T \emptyset^{(n-1)}$. Then for every weakly $2$-generic set $G$, and so certainly for every $n$-generic $G$, we have $S \oplus G \ngeq_T \emptyset'$. In particular, $S \oplus G \ngeq_T G^{(n)}$.
\end{proof}

We proceed with the proof of the theorem.

\begin{proof}[Proof of Theorem \ref{T:Cohen_gen}]
Computably in $\emptyset'$, we construct the tree $T$, an auxiliary set $B \subseteq \omega$, and a sequence $\{D_e\}_{e \in \omega}$ of dense subsets of $2^{<\omega}$. Our objective is to meet the following requirements for all $e \in \omega$.
\[
\begin{array}{lll}
\mathcal{R}_e & : & \textrm{If } S \in [T] \textrm{ and } G \subseteq \omega \textrm{ meets } D_e, \textrm{then } \Phi_e(S \oplus G) \neq B.
\end{array}
\]
Of course, every weakly $2$-generic set meets each set $D_e$. And as $B$ will be $\emptyset'$-computable, meeting these requirements suffices.

\medskip
\noindent \emph{Construction.} We obtain $T$, $B$, and each $D_e$ as $\bigcup_s T_s$, $\bigcup_s B_s$, and $\bigcup_s D_{e,s}$, where $T_s$, $B_s$, and $D_{e,s}$ denote the portions of each of these sets built by the beginning of stage $s$. Initially, let $T_0 = \{\lambda\}$, where $\lambda$ is the empty string, and let $B_0 = D_{e,0} = \emptyset$ for all $e$.

At stage $s = \seq{e,t}$, assume $T_s$, $B_s$, and $D_{e,s}$ have been defined. Let $\seq{\sigma_i : i < n}$ enumerate all strings of the form $\tau b$, where $\tau$ is a maximal string in $T_s$ and $b \in \{0,1\}$. This step adds a split above each element of $T_s$. Let $\seq{\tau_j : j < m}$ enumerate all binary strings  smaller than $s$. For all $j < m$, we will enumerate an extension $\tau_j^*$ of $\tau_j$ into $D_e$, thus ensuring that $D_e$ is dense.

For all $i<n$ and $j<m$, we will choose a unique number $x_{i,j} > s$ not in $B_s$, and define sequences
\[
\sigma_i \preceq \sigma_{i,0} \preceq \cdots \preceq \sigma_{i,m-1}
\]
and
\[
\tau_j \preceq \tau_{j,0} \preceq \cdots \preceq \tau_{j,n-1}
\]
with $|\sigma_{i,j}| = |\tau_{j,i}|$, such that one of the following applies:
\begin{enumerate}
\item\label{I:conv} there is a $b \in \{0,1\}$ such that $\Phi_e^{\sigma_{i,j} \oplus \tau_{j,i}}(x_{i,j}) \downarrow = b$;
\item\label{I:strdiv} $\Phi_e^{\sigma \oplus \tau}(x_{i,j}) \uparrow$ for all $\sigma \succeq \sigma_{i,j}$ and $\tau \succeq \tau_{j,i}$ with $|\sigma| = |\tau|$.
\end{enumerate}

We define $\sigma_{i,j}$ and $\tau_{j,i}$ simultaneously. For convenience, let $\sigma_{i,-1} = \sigma_i$ and $\tau_{j,-1} = \tau_j$. 
Let  $(i,j)$ be the lexicographically least pair in $\omega \times \omega$ such that $\sigma_{i,j}$ and $\tau_{j,i}$ are not defined (this implies that $\sigma_{i,j-1}$ and $\tau_{j,i-1}$ are already defined). 
Using $\emptyset'$, we can find a $\sigma \succeq \sigma_{i,j-1}$ and a $\tau \succeq \tau_{j,i-1}$ that satisfy \ref{I:conv} or \ref{I:strdiv} above, and we let $\sigma_{i,j}$ and $\tau_{j,i}$ be the least such $\sigma$ and $\tau$, respectively.

To complete the construction, add $x_{i,j}$ to $B_{s+1}$ for all $i,j$ such that case \ref{I:conv} above applies with $b=0$. Then, for all $j < m$, let $\tau_j^*$ be the least extension of $\tau_{j,n-1}$ greater than $s$, and add $\tau_j^*$ to $D_{e,s+1}$. Finally, for all $i < n$, let $\sigma_i^*$ be an extension of $\sigma_{i,m-1}$ of length $|\tau_j^*|$, and add all initial segments of $\sigma_i^*$ to $T_{s+1}$.

\medskip
\noindent \emph{Verification.} The construction is $\emptyset'$-computable, and so, since at each stage $s$, only elements greater than $s$ are added to $B$ and $D_e$, it follows that $B$ and the sequence $\{D_e\}_{e \in \omega}$ are computable in $A$. It is clear that the $D_e$ are dense: given any $\tau \in 2^{<\omega}$, an extension of $\tau$ is added to $D_e$ at each sufficiently large stage $s = \seq{e,t}$.

Now fix $e$, let $G$ be any weakly $2$-generic set, and let $S$ be any element of $[T]$. Then we may choose a $\tau \prec G$ in $D_e$. Let $s$ be the least stage such that $\tau \in D_{e,s}$, and let $\sigma$ be a maximal initial segment of $S$ in $T_s$, so that $|\sigma| = |\tau|$. By construction, there is an $x$ such that one of the following cases applies:
\begin{enumerate}
\item $\Phi_e^{\sigma \oplus \tau}(x) \downarrow = 1 - B(x)$;
\item $\Phi_e^{\sigma^* \oplus \tau^*}(x) \uparrow$ for all $\sigma^* \succeq \sigma$ and $\tau^* \succeq \tau$.
\end{enumerate}
Now if case \ref{I:conv} holds, then $\Phi_e^{S \oplus G}(x) \downarrow = 1 - B(x)$ since $\sigma \preceq S$ and $\tau \preceq G$. And if case \ref{I:strdiv} holds, then it cannot be that $\Phi_e^{S \oplus G}(x)$  converges, else some initial segments of $S$ and $G$ would witness a contradiction.

We conclude that $\Phi_e(S \oplus G) \neq B$, and hence that requirement $\mathcal{R}_e$ is satisfied. This completes the verification and the proof.
\end{proof}

The proof given establishes that $S \oplus G \not \ge_T \emptyset'$ by constructing an auxiliary set $B$ below $\emptyset'$ and showing that $S \oplus G \not \ge_T B$. In fact this proof  can be easily modified to show that for any non-computable set $A \le_T \emptyset'$, there is a perfect tree $T$ computable in $\emptyset'$ such that for all $S \in [T]$ and all $G \subseteq \omega$, if $G$ is weakly 2-generic then $S \oplus G \not \ge_T A$. However, this argument does not generalize directly to other forcing notions.

\section{Extensions to other forcing notions}\label{S:forcing}

We now extend Theorem \ref{T:Cohen_gen} to other forcing notions. We assume familiarity with the basics of forcing in arithmetic, but as these formulations are often dependent on the nuances of the definitions, we include a brief review of some of the particulars of our treatment. Our approach is close to that of Shore \cite[Chapter 3]{Shore-TA}, with some variations. The goal is to define forcing in a way that is  general enough to cover the forcing notions most commonly used in computability theory.

\begin{definition}\label{D:forcing}
A \emph{notion of forcing} is a triple $\mathbb{P} = (P,\leq,V)$, where:
\begin{enumerate}
\item $P$ is an infinite subset of $\omega$;
\item $\leq$ is a partial ordering of $P$;
\item $V$ is a monotone map from $(P,\leq)$ to $(2^{<\omega}, \succeq)$ such that for each $n \in \omega$, the set of $p \in P$ with $|V(p)| \geq n$ is dense.
\end{enumerate}
\end{definition}

As is customary, we refer to the elements of $P$ as the \emph{conditions} of $\mathbb{P}$, and say $q \in P$ \emph{extends} $p \in P$ if $q \leq p$. We call the map $V$ a \emph{valuation}. For each $F \subseteq P$, we let 
$|V(F)|$ denote $\sup_{p \in F}|V(p)|$.

\begin{definition}
Let $\mathbb{P} = (P,\leq,V)$ be a forcing notion, and $F$ a filter on $(P,\leq)$. If $\mathcal{C}$ is a set of subsets of $P$, then $F$ is \emph{$\mathcal{C}$-generic for $\mathbb{P}$} if
\begin{enumerate}
\item\label{I:infinity} $|V(F)| = \infty$;
\item for every $C \in \mathcal{C}$, either $F \cap C \neq \emptyset$ or $F \cap \{p \in P: (\forall q \leq p)[q \notin C]\} \neq \emptyset$.
\end{enumerate}
\end{definition}

\noindent Condition \ref{I:infinity} above ensures that if $F$ is any generic filter then $V(F)$ uniquely determines a subset of $\omega$. If $F$ is $\mathcal{C}$-generic for $\mathbb{P}$ and $G=V(F)$, then we also say that $G$ is \emph{$\mathcal{C}$-generic for $\mathbb{P}$}.

The following definition is standard in effective applications of forcing.

\begin{definition}
Let $A \subseteq \omega$ be given, and let $\mathbb{P} = (P,\leq,V)$ be a notion of forcing.
\begin{enumerate}
\item $\mathbb{P}$ is \emph{$A$-computable} if $P$, $\leq$, and $V$ are $A$-computable.
\item A set $D \subseteq P$ is \emph{$A$-effectively dense} if there is an $A$-computable function that takes each $p \in P$ to some $q \leq p$ in $D$.
\end{enumerate}
\end{definition}

We work in the usual forcing language, consisting of the language of second-order arithmetic, augmented by a new set constant $\dot{G}$ intended to denote the generic real. The (strong) forcing relation $\forces_\mathbb{P}$ is defined recursively in the standard way, starting by putting $p \forces_\mathbb{P} \varphi$ for $p \in P$  if $\varphi$ is a true atomic sentence of first-order arithmetic, or if $\varphi$ is $\dot{n} \in \dot{G}$ (respectively, $\dot{n} \not \in \dot{G}$) for some $n < |V(p)|$ such that $V(p)(n) = 1$ (respectively, $V(p)(n) =0$). For conjunctions we write $p \forces_\mathbb{P} \varphi \wedge \psi$ if $p \forces_\mathbb{P} \varphi$ and $p \forces_\mathbb{P} \psi$. For existential formulas we write $p \forces_\mathbb{P} \exists x \varphi(x)$ if $p \forces_\mathbb{P} \varphi(\dot{n})$ for some $n \in \omega$. Finally for negations we write  $p \forces_\mathbb{P} \neg \varphi$ if for all $q \leq p$ it is not true that $q \forces_\mathbb{P} \varphi$.  

It is easy to see that if $G = V(F)$ for some filter $F$ on $(P,\leq)$ with $|V(F)| = \infty$, and if $p \forces \varphi(\dot{G})$ for some $p \in F$ and some $\Sigma^0_1$ sentence $\varphi$ of the forcing language, then $\varphi(G)$ holds. (Of course, this is just a consequence of the more powerful fact that forcing implies truth for any sufficiently generic real, but we shall not need that here.)


We introduce the following concept.

\begin{definition}
Let $A \subseteq \omega$ be given,  let $\mathbb{P} = (P,\leq,V)$ be a notion of forcing, and let 
$\{\varphi_i\}_{i\in \omega}$ be an enumeration of all $\Sigma^0_1$ formulas in the forcing language. We say that  $\mathbb{P}$ is \emph{$1$-decidable in $A$} if $\mathbb{P}$ is $A$-computable, and there is an $A$-computable function 
$f: P \times \omega \rightarrow P \times \{0,1\}$ such that  if $(q, t) = f(p, i)$ then 
\begin{enumerate}
\item $q \leq p$;
\item if $t=1$ then $q \forces_\mathbb{P} \varphi_i$;
\item if $t=0$ and  $q \forces_\mathbb{P}  \neg\varphi_i$.
\end{enumerate}
\end{definition}

\noindent Note that if $A \ge_T \emptyset'$, then  being 1-decidable in $A$ reduces to the set of conditions in $\mathbb{P}$ that decide $\varphi_i$ being $A$-effectively dense, uniformly in $i$. This is because $p$ forcing a  $\Sigma^0_1$ fact corresponds a $\Sigma^0_1$ fact holding of $V(p)$, and $V(p)$ is computable in $A$.


Let $\mathbb{C}$ denote Cohen forcing. The \emph{product notion} $\mathbb{C} \times \mathbb{P}$ is the notion of forcing whose elements are pairs $(\sigma,p) \in 2^{<\omega} \times P$ such that $|\sigma| = |V(p)|$, with extension defined componentwise. The induced valuation, $V_{\times}$, on the product, is defined by
\[
V_{\times} (\sigma,p) = \sigma \oplus V(p).
\]

\noindent Cohen forcing is easily seen to be $1$-decidable in $\emptyset'$.
We wish also to calibrate the level of $1$-decidability for some other common notions of forcing, to which to apply Theorem \ref{T:nonjoin_generic} below. These include the following:
\begin{itemize}
\item Jockusch-Soare forcing inside a non-empty $\Pi^0_1$ class with no computable member;
\item Sacks forcing with perfect binary trees;
\item Mathias forcing with conditions restricted to pairs $(D,E)$ such that $E$ is an infinite computable set.
\end{itemize}
For each of these notions the set of conditions can be coded as a subset of $\omega$, and there is a natural valuation map fitting Definition \ref{D:forcing} above.
(We refer the reader to \cite[Section 8.18]{DH-2012}, \cite[Definition 2.1]{CDH-2012}, and \cite[Example 3.14]{Shore-TA},
respectively, for explicit definitions.)

In addition, we wish to consider Kumabe-Slaman forcing, which is especially important in the present discussion because of its use in proving Theorem \ref{T:Slaman_Shore}. As this forcing may be less familiar, we include its definition. (See \cite[Definition 2.5]{SS-1999}, for complete details.) A Turing functional $\Phi$, regarded as a set of triples $\seq{\sigma,x,y}$ representing that $\Phi^\sigma(x) \downarrow = y$, is called \emph{use-monotone} if the following hold:
\begin{enumerate}
\item if $\seq{\sigma,x,y}$ and $\seq{\sigma',x',y'}$ belong to $\Phi$ and $\sigma' \prec \sigma$ then $x' < x$;
\item if $\seq{\sigma,x,y} \in \Phi$ and $x' < x$, then $\seq{\sigma',x',y'} \in \Phi$ for some $y'$ and $\sigma' \preceq \sigma$.
\end{enumerate}
\begin{definition}
Let $(P,\leq)$ be the following partial order.
\begin{enumerate}
\item The elements of $P$ are pairs $(\Phi,\vec{X})$, such that $\Phi$ is a finite $\{0,1\}$-valued use-monotone Turing functional, and $\vec{X}$ is a finite set of subsets of $\omega$.
\item $(\Psi,\vec{Y}) \leq (\Phi,\vec{X})$ in $P$ if:
\begin{enumerate}
\item $\Phi \subseteq \Psi$, and if $\seq{\sigma,x,y} \in \Psi - \Phi$ and $\seq{\sigma',x',y'} \in \Phi$ then $|\sigma| > |\sigma'|$;
\item $\vec{X} \subseteq \vec{Y}$, and if $\sigma$ is an initial segment of some $X \in \vec{X}$ and $\seq{\sigma,x,y} \in \Psi$ then $\seq{\sigma,x,y} \in \Phi$.
\end{enumerate}
\end{enumerate}
\end{definition}

\noindent Kumabe-Slaman forcing is the notion $(P,\leq,V)$ with $P$ and $\leq$ as above, and $V : P \to 2^{<\omega}$ defined by $V((\Phi,\vec{X})) = \Phi$. Of course, $P$ here cannot be  coded as a subset of $\omega$. However, our interest will be in a restriction of this forcing that can be so coded, namely, when the conditions are pairs $(\Phi,\vec{X})$ such that $\vec{X}$ consists of $\emptyset^{(n)}$-computable sets.

We have the following bounds on the complexity of each of these notions.

\begin{proposition}\label{P:decidability}
\
\begin{enumerate}
\item\label{I:Cohen} Cohen forcing is $1$-decidable in $\emptyset'$.
\item\label{I:JS} Jockusch-Soare forcing inside a non-empty $\Pi^0_1$ class with no computable member is $1$-decidable in $\emptyset'$.
\item\label{I:S} Sacks forcing with perfect trees is $1$-decidable in $\emptyset''$.
\item\label{I:Mathias} Mathias forcing with computable sets is $1$-decidable in $\emptyset''$.
\item\label{I:KS} Kumabe-Slaman forcing with $\emptyset^{(n)}$-computable sets is $1$-decidable in $\emptyset^{(n+2)}$.
\end{enumerate}
\end{proposition}

\begin{proof}
Part \ref{I:Cohen} is clear. Part \ref{I:JS} is essentially by the proof of the low basis theorem, as forcing the jump involves asking whether a particular $\Pi^0_1$ subclass is non-empty. Part \ref{I:S} is similar, but we gain an extra quantifier as the set of conditions is only computable in $\emptyset''$.
Part \ref{I:Mathias} follows by Lemma 4.3 of \cite{CDH-2012}, and \ref{I:KS} is implicit in Lemmas 2.10 and 2.11 of \cite{SS-1999}.
\end{proof}

Note that for any of the notions $\mathbb{P}$ in the previous proposition, the same bounds on $1$-decidability apply also to $\mathbb{C} \times \mathbb{P}$.

We can now state and prove our extension of Theorem \ref{T:Cohen_gen}.

\begin{theorem}\label{T:nonjoin_generic}
Let $A \subseteq \omega$ be given, and let $\mathbb{P} = (P,\leq,V)$ be a notion of forcing such that $\mathbb{C} \times \mathbb{P}$ is $1$-decidable in $A$. There exists an $A$-computable perfect tree $T \subseteq 2^{<\omega}$, and an $A$-computable class $\mathcal{C}$ of dense subsets of $P$, such that for all $S \in [T]$ and all $G \subseteq \omega$, if $G$ is $\mathcal{C}$-generic then $S \oplus G \ngeq_T A$.
\end{theorem}

\begin{proof}
The proof is similar to that of Theorem \ref{T:Cohen_gen}, so we just highlight the differences. The construction is now computable in $A$, the sets $D_e$ are subsets of $P$, and we let $\mathcal{C} = \{D_e\}_{e \in \omega}$. The requirements take the following form.
\[
\begin{array}{lll}
\mathcal{R}_e & : & \textrm{If } S \in [T] \textrm{ and } F \textrm{ is a filter on } (P,\leq) \textrm{ with } |V(F)| = \infty \textrm{ that meets } D_e,\\
& & \textrm{then } \Phi_e(S \oplus V(F)) \neq B. 
\end{array}
\]
The construction is modified in that at stage $s = \seq{e,t}$ we fix an enumeration $\seq{p_j : j < m}$ of all conditions smaller than $s$, and instead of building two sequences of strings, $\sigma_i \preceq \sigma_{i,0} \preceq \cdots \preceq \sigma_{i,m-1}$ and $\tau_j \preceq \tau_{j,0} \preceq \cdots \preceq \tau_{j,n-1}$, the latter is replaced by a sequence $p_j \geq p_{j,0} \geq \cdots \geq p_{j,n-1}$ of conditions. For all $i,j$, we ensure that $|\sigma_{i,j}| = |V(p_{j,i})|$ and that one of the following applies:
\begin{enumerate}
\item $(\sigma_{i,j},p_{i,j}) \forces_{\mathbb{C} \times \mathbb{P}} \Phi_e(\dot{G})(x_{i,j}) \downarrow = 1$;
\item $(\sigma_{i,j},p_{i,j}) \forces_{\mathbb{C} \times \mathbb{P}} \neg(\Phi_e(\dot{G})(x_{i,j}) \downarrow = 1)$.
\end{enumerate}
This can be done $A$-computably, using the fact that $\mathbb{C} \times \mathbb{P}$ is $1$-decidable in $A$ to decide the $\Sigma^0_1$ sentence $\Phi_e(\dot{G})(x_{i,j}) \downarrow = 1$. The definitions of $T_{s+1}$, and $D_{e,s+1}$, are then entirely analogous to their definitions in the proof of Theorem \ref{T:Cohen_gen}, and $B_{s+1}$ is obtained from $B_s$ by adding $x_{i,j}$ for all $i,j$ for which case \ref{I:strdiv} applies.

For the verification, fix $e$, let $G \subseteq \omega$ be $\mathcal{C}$-generic, and let $S$ be any element of $[T]$. Let $F$ be a filter on $(P,\leq)$ such that $G = V(F)$, so that $|V(F)|=\infty$ and $F$ meets $D_e$. Choose $p \in F \cap D_e$, let $s$ be the least stage such that $p \in D_{e,s}$, and let $\sigma$ be a maximal initial segment of $S$ in $T_s$, so that $|\sigma| = |V(p)|$. By construction, there is an $x$ such that one of the following cases applies:
\begin{enumerate}
\item $B(x) = 0$ and $(\sigma,p) \forces_{\mathbb{C} \times \mathbb{P}} \Phi_e(\dot{G})(x) \downarrow = 1$;
\item $B(x) = 1$ and $(\sigma,p) \forces_{\mathbb{C} \times \mathbb{P}} \neg(\Phi_e(\dot{G})(x) \downarrow = 1)$.
\end{enumerate}
If case \ref{I:conv} holds, then by the general definition of forcing, $\Phi_e(V_{\times}(\sigma,p))(x) \downarrow = 1$, and hence $\Phi_e(\sigma \oplus V(p))(x) \downarrow = 1 \neq B(x)$. It follows that $\Phi_e(S \oplus V(F))(x) \downarrow = 1 - B(x)$ since $\sigma \prec S$ and $V(p) \prec V(F)$. Now suppose case \ref{I:strdiv} holds. If $\Phi_e(S \oplus V(F))(x) \downarrow = 1$ then there exists some $n \geq |\sigma| = |V(p)|$ such that $\Phi_e(S \res n \oplus V(F) \res n)(x) \downarrow = 1$. As $|V(F)| = \infty$, $F$ necessarily contains some $q \in P$ with $|V(q)| \geq n$, and as $F$ is a filter, we may assume $q \leq p$. Let $\tau$ be any initial segment of $S$ of length $|V(q)|$. Then $(\tau,q)$ is an extension of $(\sigma,p)$ in $\mathbb{C} \times \mathbb{P}$ that forces $\Phi_e(\dot{G})(x) \downarrow = 1$, which is impossible. Hence, it must be that $\Phi_e(S \oplus V(F))(x) \uparrow$ or $\Phi_e(S \oplus V(F))(x) \downarrow = 0$.
\end{proof}

Now suppose $\mathbb{P}$ is one of the notions of forcing listed above. We obtain the following consequences.

\begin{corollary}\label{C:all_notions}
Let $\mathbb{P} = (P,\leq,V)$ be any of the notions of forcing discussed above, and let $A \subseteq \omega$ be such that $\mathbb{P}$ is $1$-decidable in $A$, as calibrated in Proposition \ref{P:decidability}. (Hence $\mathbb{C} \times \mathbb{P}$ is $1$-decidable in $A$, as remarked above.) There exists an $A$-computable class $\mathcal{C}$ of dense subsets of $P$, and for each non-computable set $B \subseteq \omega$ an $S \nleq_T B$, such that for all $G \subseteq \omega$, if $G$ is $\mathcal{C}$-generic for $\mathbb{P}$ then $S \oplus G \ngeq_T A$.
\end{corollary}

\begin{proof}
The proof is similar to that of Corollary \ref{C:Cohen_gen}, which is just the special case when $\mathbb{P} = \mathbb{C}$, $A = \emptyset'$, and $B = \emptyset^{(n-1)}$. Theorem \ref{T:nonjoin_generic} produces the class $\mathcal{C}$, as well as a perfect tree $T$ through which we choose an infinite path $S \nleq_T B$. By construction, $S \oplus G \ngeq_T A$ for all $\mathcal{C}$-generic $G$.
\end{proof}

We conclude this section by noting an interesting consequence of the corollary to Kumabe-Slaman forcing. As noted above, general Kumabe-Slaman forcing, i.e., forcing with pairs $(\Phi,\vec{X})$, where $\vec{X}$ ranges over arbitrary finite collections of sets, cannot be readily coded as a subset of $\omega$. One could nonetheless ask whether some such coding is possible, and in particular, whether the notion can be made $1$-decidable in some arithmetical set $A$. The corollary implies that this cannot be so. Indeed, the proof of Theorem \ref{T:Slaman_Shore} in \cite{SS-1999} for $n$ only uses Kumabe-Slaman with $\emptyset^{(n)}$-computable sets, and produces a set $G = V(F)$ for a filter $F$ that decides every $\Sigma^0_n$ sentence. It is not difficult to see that if $n \geq m+1$ then any such $G$ is generic for every $\emptyset^{(m)}$-computable class of subsets of conditions. So, if Kumabe-Slaman forcing were $1$-decidable in $\emptyset^{(m)}$, we could apply the corollary to find an $S \nleq_T \emptyset^{(n-1)}$ such that for all $G$ as above, $S \oplus G \ngeq_T \emptyset^{(m)}$, and hence certainly $S \oplus G \ngeq_T G^{(n)}$, giving a contradiction.

\section{A non-joining theorem for randoms}\label{S:randoms}

We now ask  whether the set $G$ of the Shore-Slaman theorem, Theorem \ref{T:Slaman_Shore}, can be chosen to be $n$-random. Randomness and genericity are each a notion of typicality, so a negative answer would suitably complement Corollary \ref{C:all_notions}. In this section, we show that the answer is indeed negative.

Although randomness and genericity are orthogonal notions in most respects, it is possible to think of randomness as a notion of genericity in a limited way. Namely, by a result of Kautz, the weakly $n$-random sets  can be characterized in terms of genericity for Solovay forcing with $\Pi^0_n$ classes of positive measure (see \cite[Section 7.2.5]{DH-2012}). 
While it may at first seem possible to obtain the result for $n$-randomness as just another application of Theorem \ref{T:nonjoin_generic}, it is worth pointing out that this is not the case. The reason is that forcing with $\Pi^0_n$ classes of positive measure is only $1$-decidable in $\emptyset^{(n+2)}$, which means that Theorem \ref{T:nonjoin_generic}  requires a higher level of genericity than is used in Kautz's characterization. This level of genericity no longer corresponds to $n$-randomness.

Our goal, then, is to give a direct argument for the following result, which is analogous to Theorem \ref{T:Cohen_gen}.

\begin{theorem}\label{T:nonjoin_random}
For any non-computable set $A \leq_T \emptyset'$, 
there exists a $\emptyset'$-computable perfect tree $T \subseteq 2^{<\omega}$ such that for all $S \in [T]$ and all $R \subseteq \omega$, if $R$ is $2$-random then $S \oplus R \ngeq A$.
\end{theorem}

The following corollary is then obtained exactly as Corollary \ref{C:Cohen_gen} was above.

\begin{corollary}
For all $n \geq 2$, there exists an $S \nleq_T \emptyset^{(n-1)}$ such that for all $G \subseteq \omega$, if $G$ is $2$-random then $S \oplus G \ngeq_T \emptyset'$. In particular, $S \oplus G \ngeq_T G^{(n)}$.
\end{corollary}

To prove the theorem, we need the following lemma. Fix a Turing reduction $\Phi$, and define $m(\tau,\rho) = \mu(\{X :\Phi(\tau \oplus X) [2\cdot |\tau|] \succeq \rho\})$ for all $\tau,\rho \in 2^{<\omega}$, where $\mu$ is the uniform measure on Cantor space.

\begin{lemma}\label{L:nondensity}
Let $A \subseteq \omega$ be non-computable. For all rational $q > 0$ and all $\sigma \in 2^{<\omega}$, there is an $n \in \omega$ such that the set of $\tau$ with $m(\tau, A \res n) \geq q$ is not dense above $\sigma$.
\end{lemma}

\begin{proof}
Assume not, and let $q$ and $\sigma$ witness this fact. Let $G \subseteq \omega$ be any set which is $1$-generic relative to $A$ and extends $\sigma$. Then $G$ meets $\{\tau \succeq \sigma : m(\tau,A \res n) \geq q\}$ for all $n$, and hence if we let $E_n =\{X \subseteq \omega: \Phi(G \oplus X) \succeq A \res n\}$, then $\mu (E_n) \geq q$ for all $n$. Let $E = \bigcap_n E_n$. It follows that $\mu (E) \geq q$, and $\Phi(G\oplus X) = A$ for all $X \in E$. By a theorem of Sacks \cite{Sacks-1963}, this implies that $G \geq_T A$, which is impossible since anything $1$-generic relative to a non-computable set cannot compute that set.
%
\end{proof}

\begin{proof}[Proof of Theorem \ref{T:nonjoin_random}]
Because the 2-random sets are closed under the addition of prefixes, we need only consider a single functional $\Phi$ such that $\Phi(X \oplus 0^e1Y) = \Phi_e(X \oplus Y)$ for all sets $X$ and $Y$ and all $e \in \omega$. Assume the function $m$ from above is defined with respect to this $\Phi$.

Computably in $\emptyset'$, we construct the tree $T$ by stages, along with a sequence $\seq{n_i : i \in \omega}$ of numbers. Our test will be defined by  
\[
U_i = \{ \tau \in 2^{<\omega} : (\exists \sigma \in T)[ \Phi(\sigma \oplus \tau) \succeq A \res n_i ]\}.
\]
Clearly, then, $\{U_i\}_{i \in \omega}$ will be a $\emptyset'$-computable sequence of open sets, and if $S \in [T]$ and $R \subseteq \omega$ is such that $\Phi(S \oplus R) =A$, then some initial segment of $R$ will belong to every $U_i$. Thus, to verify the construction below, it will suffice to ensure that $\mu (U_i) \le 2^{-i}$ for all $i$.  

\medskip
\noindent \emph{Construction.}
Let $T_s$ denote the approximation to $T$ at stage $s$, with $T_0 = \{\lambda\}$. At stage $s$, we define  $T_{s+1}$ and $n_s \in \omega$. Let $\seq{\sigma_i : i < m}$ be a listing of the maximal strings in $T_s$. We need to add a split above each of these strings, but before can do so, we will make a series of extensions to each $\sigma_i$. 

First, we  initialize level $s$ of our test. To do this we find $n_s \in \omega$, and for all $i$, a string $\sigma_i^*$ extending $\sigma_i$ such that $m(\sigma, A\res n_s) \leq 2^{-2(s+1)}$ for all $\sigma \succeq \sigma_i^*$.   Lemma~\ref{L:nondensity} establishes that  there exists an $n_s$ and  sequence $\seq{\sigma_i^* : i<m}$ meeting this condition, and an instance can be found using $\emptyset'$.

Now for any path extending $\sigma_i^*$ we have bounded the size of the strings that will be enumerated into $U_s$ due to this path. However, there will be continuum many paths above $\sigma_i^*$ in $T$ so we need to do better than this. For all $t \le s$, we extend each $\sigma_i^*$ in order to force a large number of elements into $U_t$. We determine a sequence of extensions 
\[
\sigma_i^* \preceq \sigma_{i,0} \preceq \ldots  \preceq \sigma_{i,s}
\]
inductively as follows. Suppose we have defined $\sigma_{i,t-1}$ for some $t \leq s$, where for convenience we write $\sigma_{i,-1} = \sigma_i^*$. Then we find $\sigma_{i,t} \succeq \sigma_{i,t-1}$ such that
\begin{equation}\label{E:supremum}
  \sup_{\sigma \succeq \sigma_{i,t-1}}m(\sigma, A \res n_{t}) - m(\sigma_{i,t}, A \res n_{t}) \leq 2^{-2(s+2)}.
\end{equation}

We define $T_{s+1}$ to be downward closure of $\{\sigma_{i,s}b : i < m \wedge b \in \{0,1\}\}$.

\medskip
\noindent \emph{Verification.}
For each $i$ and $s$, define
\[
U_{i,s} = \{ \tau \in 2^{<\omega} : (\exists \sigma \in T_{i+s+1})[ \Phi(\sigma\oplus \tau) \succeq A \res n_i ]\}.
\]
As $T = \bigcup_s T_s$, and $T_s \subseteq T_{s+1}$, we have $U_i = \bigcup_s U_{i,s}$. Note that for all $s$, the measure of $U_{i,s}$ is equal to the sum over the maximal $\sigma \in T_{s+1}$ of $m(\sigma,A \res n_i)$. We prove by induction on $j \in \omega$ that $\mu(U_{i,j}) \leq \sum_{k \leq j} 2^{-(i+k+1)}$, and hence that $\mu(U_i) \leq 2^{-i}$, as desired.

For the base case $j = 0$, the claim holds because there are $2^{i+1}$ many maximal strings in $T_{i+1}$, and for any such string $\sigma$ we have $m(\sigma,\emptyset' \res n_i) < 2^{-2(i+1)}$ by choice of $n_i$. Assume, then, that $j > 0$ and that the claim holds for $j-1$. Let $\sigma$ be a maximal string in $T_{i+j+1}$, and let $\tau \preceq \sigma$ be a maximal string in $T_{i+j}$. By \eqref{E:supremum}, we have that $m(\tau, A\res n_i)$ is within $2^{-2(i+j+1)}$ of $\sup_{\rho \succeq \tau} m(\rho,A \res n_i)$. Thus, the total new contribution of $m(\sigma, A \res n_i)$ to the measure of $U_{i,j}$ can be at most $2^{-2(i+j+1)}$. Since there are $2^{i+j+1}$ many maximal strings in $T_{i+j+1}$, the claim follows.
\end{proof}

For the readers familiar with the basic concepts of algorithmic randomness, we remark that if $T$ is the tree constructed in Theorem \ref{T:nonjoin_random}, then it is easy to see that $\emptyset'$ can construct a set $S\in [T]$ such that $S$ is not $K$-trivial. This gives an example of a set that is not $K$-trivial and cannot be joined above $\emptyset'$ with a 2-random set. This  contrasts with a recent result of Day and Miller \cite{DM-toappear} who show that any set $S$ that is not $K$-trivial can be joined above $\emptyset'$ with an incomplete 1-random set, and indeed even with a weakly 2-random set.

\end{document}